\documentclass{aomart}
\usepackage{amssymb,amsmath,mathrsfs,amsthm, mathtools}
\usepackage[english]{babel}
\usepackage{enumitem}
\usepackage[dvipsnames]{xcolor}
\usepackage{url}
\usepackage{tikz-cd}
\usepackage{caption}
\usepackage{graphicx}
\usepackage[framemethod=tikz]{mdframed}
\usepackage[printwatermark]{xwatermark}

\def\C{\mathsf{C}}
\def\Cpt{\mathsf{C}_{\ast}}
\def\top{\mathsf{Top}}
\def\set{\mathsf{Set}}
\def\setpt{\mathsf{Set}_{\ast}}
\def\toppt{\mathsf{Top}_{\ast}}
\def\epitop{\mathsf{EpiTop}}
\def\epitoppt{\mathsf{EpiTop}_{\ast}}

\def\pstop{\mathsf{PsTop}}
\def\pstoppt{\mathsf{PsTop}_{\ast}}

\def\Lim{\mathsf{Lim}}
\def\conv{\mathsf{Conv}}
\def\kconv{\mathsf{KConv}}
\def\pstopgrp{\mathsf{PsTopGrp}}
\def\epitopgrp{\mathsf{EpiTopGrp}}
\def\topgrp{\mathsf{TopGrp}}
\def\grp{\mathsf{Grp}}
\def\qtopgrp{\mathsf{QTopGrp}}

\newcommand{\inclu}{\hookrightarrow}

\DeclareMathOperator{\ev}{\mathrm{ev}}

\DeclareMathOperator{\biss}{\pi_1^{\mathsf{qtop}}}
\DeclareMathOperator{\brazas}{\pi_1^{\mathsf{top}}}
\DeclareMathOperator{\pips}{\pi_1^{\mathsf{ps}}}
\DeclareMathOperator{\piepi}{\pi_1^{\mathsf{epi}}}

\DeclareMathOperator{\pihps}{\pi_\mathnormal{n}^{\mathsf{ps}}}
\DeclareMathOperator{\pihepi}{\pi_\mathnormal{n}^{\mathsf{epi}}}

\newcommand{\Prod}{\operatorname{\Pi}}

\DeclareMathOperator{\pcset}{\pi_0}

\DeclareMathOperator{\pctop}{\pi_0}

\DeclareMathOperator{\pcepi}{{\pi}_0^\mathsf{epi}}

\DeclareMathOperator{\pcps}{{\pi}_0^\mathsf{ps}}

\DeclareMathOperator{\looptop}{\Omega}

\DeclareMathOperator{\loopepi}{\Omega^\mathsf{epi}}

\DeclareMathOperator{\loopps}{\Omega^\mathsf{ps}}

\DeclareMathOperator{\loopC}{\Omega^\mathsf{C}}

\newcommand{\quot}[2]{q_{#1}^{\mathrm{#2}}}
\newcommand{\quott}[2]{p_{#1}^{\mathrm{#2}}}

\newcommand{\id}{\mathrm{id}}
\newtheorem{defn}{Definition}[section]

\newtheorem{prop}[defn]{Proposition}
\newtheorem{lemma}[defn]{Lemma}
\newtheorem{cor}[defn]{Corollary}

\newtheorem{rmk}[defn]{Remark}

\newtheorem{problem}{Problem}

\surroundwithmdframed[%
					roundcorner		=	5pt		,%
					innertopmargin		=	0pt		,%
					innerbottommargin	=	5pt		,%
					skipabove			=	4pt		,%
					skipbelow			=	4pt
					]{problem}

\title[Epi- and pseudo-topological fundamental group functors]{Epitopological and pseudotopological fundamental group functors}

\author{Giacomo Dossena}
\email{giacomo.dossena@gmail.com}

\thanks{March 2015}
\copyrightnote{}

\begin{document}

\begin{abstract}
In these notes the epitopological and pseudotopological fundamental group functors are introduced. These are functors from the category of pointed epitopological and pseudotopological spaces respectively, to the category of their respective group-objects. Their restrictions to the full subcategory of topological spaces are lifts of the topologized fundamental group functor introduced by Daniel Biss in \cite{Bis2002} and thus retain its information. At the same time, they show greater regularity inherited from the convenient properties of $\epitop$ and $\pstop$. Moreover, the use of such convenient categories permits, in principle, to apply general techniques from enriched homotopy theory. Our approach should be compared with the alternative improvement of Biss's functor developed by Jeremy Brazas in \cite{Bra2013} within the topological setting. Several open problems, including those aimed at understanding the precise relation to Brazas's approach, are scattered throughout the text.
\end{abstract}

\maketitle
\thispagestyle{plain}

\tableofcontents

\section{Overview}\label{sec:overview}
The usual fundamental group functor $\pi_1\colon\toppt \to\grp$ which assigns to each pointed topological space $(X,x_0)\in\toppt$ the group of homotopy classes of based loops and to each based continuous map $f\colon X\to Y$ the group homomorphism $[l]\mapsto [f\circ l]$ can be factorized as $\pi_1=\pcset\circ\looptop$ where $\looptop\colon\toppt\to\toppt$ is the loop space functor and $\pcset\colon\toppt\to\setpt$ is the path-component set functor. It is evident that the latter forgets the topological structure of $\looptop(X,x_0)$, thus throwing away potentially useful information encoded in it. A simple workaround to make this information available is to lift $\pcset$ to a functor $\pctop\colon\toppt\to\toppt$ (we keep the same symbol for economical reason) by placing on the set of path components of a space the quotient topology induced by the natural projection that assigns to each point its path component. This was done by Daniel Biss \cite{Bis2002} and results in a topologized fundamental group functor $\biss\colon\toppt\to\qtopgrp$, where $\qtopgrp$ is the category of quasitopological groups\footnote{A quasitopological group is almost like a topological group, except that we do not require multiplication to be continuous.} and continuous homomorphisms. As clarified by Paul Fabel \cite{Fab2009} and Jeremy Brazas \cite{Bra2011}, the reason why $\biss$, contrary to what is claimed in \cite{Bis2002}, does not land in $\topgrp$ is intimately related to the fact that in $\top$ a product of quotient maps need not be quotient. This shortcoming has inspired an ingenious modification \cite{Bra2013} of the quotient topology of $\biss(X,x_0)$ to obtain a genuine $\topgrp$-valued functor $\brazas$, consisting in iterating (transfinitely many times, in general) the procedure that places on $\pi_1(X,x_0)$ the quotient topology with respect to the multiplication $\biss(X,x_0)\times\biss(X,x_0)\to\pi_1(X,x_0)$. The additional information encoded in $\biss$ and $\brazas$ is inconspicuous for spaces admitting universal covers, whereas it turns out to be useful to discriminate among spaces with complicated local structure, i.e.~that fail to be locally path connected or semilocally simply connected. The functor $\brazas$ allows also for a generalized theory of covering spaces \cite{Bra2012a}.

We would like to follow a different route and leave $\top$ in favor of one of its better-behaved supercategories where quotient maps are product-stable, thus yielding a continuous multiplication in the fundamental group. Clearly the new functor will not land in $\topgrp$ but rather in the category of group-objects of the chosen supercategory.

When substituting $\top$ with a larger category there are two contrasting principles at work \cite{Her1987} summarized by the catchlines ``the ampler the better'' and ``the meagerer the better'', that is, larger extensions are needed when stronger convenience requirements are made, whereas smaller extensions generally retain more structure of the original category. In our case it turns out there are at least two suitable extensions of $\top$: the smaller one is its cartesian closed topological hull $\epitop$, whose objects are the so-called epitopological spaces (also known as Antoine spaces), and the larger one is its topological universe hull $\pstop$, whose objects are the so-called pseudotopological spaces.

In the spirit of the catchlines above, $\epitop$ and $\pstop$ are large enough to be cartesian closed, and small enough to be still significantly related to $\top$. The key fact that makes $\epitop$ and $\pstop$ suited for defining our enriched fundamental groups is the validity of the pasting lemma, together with the concrete reflectivity of the embeddings $\top\inclu\epitop\inclu\pstop$. We shall give proper definitions in the next sections. Here we only mention that concrete reflectivity does not rule out larger constructs from the list of possibly useful extensions, whereas the validity of the pasting lemma seems to be peculiar to extensions not larger than $\pstop$.

In common with $\biss$ and $\brazas$, the resulting functors
\begin{equation*}
\begin{aligned}
\piepi&\colon\epitoppt\to\epitopgrp\\
\pips&\colon\pstoppt\to\pstopgrp
\end{aligned}
\end{equation*}
are homotopy invariant and behave as expected under a change of basepoint. Moreover, their restrictions to $\toppt$ are suitable lifts of $\biss$. This implies that all the information encoded in $\biss$ is available in both $\piepi$ and $\pips$.
A piece of evidence that convenient properties of $\epitop$ and $\pstop$ translate into greater regularity of $\piepi$ and $\pips$ than $\biss$ is given by the fact that both $\piepi$ and $\pips$ preserve finite products in their respective categories ($\pips$ preserves even arbitrary ones), while $\biss$ does not. This should be compared with the fact that $\brazas$ preserves finite products (it is not known whether it preserves arbitrary ones). A thorough investigation of basic categorical properties of $\piepi$ and $\pips$ has yet to be done.

\section{\texorpdfstring{$\pstop$}{PsTop} and \texorpdfstring{$\epitop$}{EpiTop}}\label{sec:presentation}
No originality is claimed for this section. The reader familiar with these categories may safely skip it. References for $\pstop$ include \cite{Cho1947}, \cite{BenHerLow1991}, \cite{HerColSch1991}, \cite{Wyl1991}. References for $\epitop$ include \cite{Ant1966}, \cite{Mac1973}, \cite{Bou1975}.

The existence of two different topologized versions of $\pi_1$ as reviewed in Sec.\@ \ref{sec:overview} can be ascribed to the fact that $\top$ is not cartesian closed. Indeed, if $\top$ were cartesian closed then the two procedures that define $\biss$ and $\brazas$ would give the same topologization for $\pi_1(X,x_0)$. In particular, this topologization would enjoy some desirable properties associated with $\biss$ and $\brazas$ separately: it would be defined in a straightforward way internally, as in $\biss$, and it would make the corresponding functor land in the category of group-objects of $\top$, as in $\brazas$. It is then natural to construct a variant of these functors on a cartesian closed category related to $\top$. We might consider \emph{sub}categories of $\top$ but, in order to keep all topological spaces into the picture, we consider instead its \emph{super}categories. This has the additional effect of recovering $\biss$ by reflecting back to $\top$, as we shall see.
For reasons already explained in Sec.\@ \ref{sec:overview} we look for extensions that do not depart too much from $\top$. A useful setting for such a task is that of topological constructs. Loosely speaking, these are categories of structured sets admitting initial and final structures, as $\top$. More precisely, a topological construct is a concrete category over $\set$ which has unique initial (and hence also unique final) structures with respect to the forgetful functor (see the full treatment \cite{AdaHerStr2004} for more details). Sometimes a topological construct is assumed to be well-fibered\footnote{Well-fibered means that the following two conditions hold: for each set $X$, the class of objects having $X$ as underlying set is a set; for each set $X$ with at most one element, there is exactly one object with underlying set $X$.}. All the constructs considered in these notes are well-fibered, so in order to facilitate the presentation we tacitly assume a topological construct to be such.

We now recall the notion of an exponentiable object in a topological category. First notice that any topological construct $\C$ has concrete products: the product of a family of $\C$-objects $(X_j)_{j\in J}$ is defined as the unique initial lift of the structured source $(\Prod_{k\in J} |X_k| \to X_j)_{j\in J}$. Moreover, given $\C$-objects $X,Y,Z$ and a $\C$-morphism $f\colon X\times Y\to Z$, for each point $x\in X$ the set map $f_x\colon Y\to Z$ given by $f_x(y)=f(x,y)$ is a $\C$-morphism\footnote{Here we used the assumption of well-fiberedness.}.

\begin{defn}
Given a topological construct $|\cdot|\colon\C\to\set$, an object $X\in\C$ is said to be exponentiable if for each $Y\in\C$ there is a $\C$-object, denoted by $Y^X$, such that:
\begin{enumerate}
\item $|Y^X|=\C(X,Y)$,
\item the evaluation map $\ev\colon Y^X\times X\to Y$ is a $\C$-morphism,
\item for any $Z\in\C$ and any $\C$-morphism $h\colon Z\times X\to Y$ the corresponding set map $\hat{h}\colon Z\to Y^X$ defined by $\hat{h}(z)(x)=h(z,x)$ is a $\C$-morphism.
\end{enumerate}
The $\C$-structure defining $Y^X$ is called exponential.
\end{defn}

\begin{rmk}
Notice that for an exponentiable object $X\in\C$ the correspondence $\C({Z\times X}, Y)\to\C(Z,Y^X)$ given by $h\mapsto \hat{h}$ is a bijection.
\end{rmk}

A topological construct is cartesian closed iff all its objects are exponentiable (\cite{Her1974} or \cite[Thm 2.14]{HerColSch1991}). Although $\top$ is not cartesian closed, the class of its exponentiable spaces is well understood \cite{EscHec2001} and contains all locally compact Hausdorff spaces. In particular, the interval $[0,1]$ with its standard topology is exponentiable. Moreover, for a locally compact Hausdorff space $X$ and any space $Y\in\top$ the exponential topology defining $Y^X$ is the compact-open topology.

Before introducing pseudotopological spaces we briefly recall that, given a map $f\colon X\to Y$ between sets and a filter $\mathscr{F}$ on $X$, we can always define the pushforward filter $f_*\mathscr{F}$ on $Y$ as the filter generated by the filter base $\{f(F)\mid F\in\mathscr{F}\}$. It turns out that $f_*\mathscr{F}=\{S\subset Y\mid f^{-1}(S)\in \mathscr{F}\}$ and that for any maps $g\colon X\to Y$ and $f\colon Y\to Z$ we have $(f\circ g)_*\mathscr{F}=f_*(g_*\mathscr{F})$. Moreover, if $\mathscr{F}$ is a (principal) ultrafilter then $f_*\mathscr{F}$ is a (principal) ultrafilter. For a set $X$ and $x\in X$ we denote by $F(X)$ the set of all filters on $X$, by $U(X)$ the set of all ultrafilters on $X$ and by $\dot{x}$ the principal ultrafilter on $X$ consisting of all supersets of $\{x\}$.

\begin{defn}[$\pstop$]
A pseudotopological structure on a set $X$ is a relation $u\subset U(X)\times X$ such that $(\dot{x},x)\in u$ for each $x\in X$. The pair $(X,u)$ is called a pseudotopological space, or pseudospace for short. A map $f\colon X\to Y$ between pseudospaces $(X,u_X)$ and $(Y,u_Y)$ is called continuous at $x\in X$ if we have $(f_*\mathscr{U},f(x))\in u_Y$ whenever $(\mathscr{U},x)\in u_X$. A map $f\colon X\to Y$ is called continuous if it is continuous at each $x\in X$. We denote by $\pstop$ the category whose objects are all pseudospaces and whose morphisms are all continuous maps between them.
\end{defn}

The set of all pseudotopological structures on a given set is a complete lattice with respect to set inclusion. The least and greatest elements are, respectively, the discrete pseudotopological structure $\{(\dot{x},x)\mid x\in X\}$ and the indiscrete pseudotopological structure $U(X)\times X$. Given a set map $f\colon X\to Y$, a pseudospace structure $u$ on $X$ induces the pseudotopological structure $\{(f_*\mathscr{U},f(x))\mid (\mathscr{U},x)\in u\}\cup\{(\dot{y},y)\mid y\in Y\}$ on $Y$ and this is the least one among all pseudotopological structures on $Y$ making $f$ continuous. It is the final pseudotopological structure with respect to the pseudospace $(X,u)$ and the set map $f$. Dually, a pseudospace structure $u$ on $Y$ induces the pseudotopological structure $\{(\mathscr{U},x)\mid(f_*\mathscr{U},f(x))\in u\}\cup\{(\dot{x},x)\mid x\in X\}$ on $X$ and this is the greatest one among all pseudotopological structures on $X$ making $f$ continuous. It is the initial pseudotopological structure with respect to the pseudospace $(Y,u)$ and the set map $f$. It follows that the category $\pstop$ together with the obvious forgetful functor, defined by $|(X,u)|=X$ on objects and the identity on maps, is a topological construct. Explicitly, the initial pseudotopology $u$ on a set $X$ with respect to a collection of maps $\{f_j\colon X\to (X_j,u_j)\}_{j\in J}$ is the intersection of all the initial pseudotopological structures induced by each $f_j$. Dually, the final pseudotopology $u$ on a set $X$ with respect to a collection of maps $\{f_j\colon (X_j,u_j)\to X\}_{j\in J}$ is the union of all the final pseudotopological structures induced by each $f_j$.

In the following we often use the alternative notation $\mathscr{U}\to_u x$, or more simply $\mathscr{U}\to x$, to indicate $(\mathscr{U},x)\in u$. When there is no possibility of confusion we adopt the abuse of notation of writing $X$ for the pseudospace $(X,u)$ and, accordingly, sometimes we use the notation $\mathscr{U}\to_X x$.

A pseudotopological structure $u\subset U(X)\times X$ induces a relation $u'\subset F(X)\times X$ defined by $(\mathscr{F},x)\in u'$ iff for all ultrafilters $\mathscr{U}\supset \mathscr{F}$ we have $(\mathscr{U},x)\in u$. When we write $\mathscr{F}\to x$ for a filter $\mathscr{F}$ on a pseudospace $(X,u)$ we thus mean\footnote{It is actually possible to define a pseudotopological structure directly by means of filters rather than ultrafilters, thereby using a relation $v\subset F(X)\times X$, in which case we must append to the single axiom $\dot{x}\to_v x$  the following two additional axioms:
\begin{itemize}
\item $F(X)\ni\mathscr{G}\supset\mathscr{F}\to_v x\implies\mathscr{G}\to_v x$,
\item if a filter $\mathscr{F}\not\to_v x$ then there is a filter $\mathscr{F}'\supset \mathscr{F}$ such that, for all filters $\mathscr{F}''\supset \mathscr{F}'$, $\mathscr{F}''\not\to_v x$.
\end{itemize}} $(\mathscr{F},x)\in u'$.

Not surprisingly, $\top$ can be viewed as a full isomorphism-closed subcategory of $\pstop$ by assigning to each topological space $X$ the pseudospace consisting of the same underlying set equipped with the pseudotopological structure given by ultrafilter convergence.

As opposed to $\top$, the construct $\pstop$ is cartesian closed \cite{HerColSch1991}. For $X,Y\in\pstop$ the exponential pseudotopology on $\pstop(X,Y)$ is given by declaring, for a filter $\mathscr{F}$ on $\pstop(X,Y)$ and $f\in\pstop(X,Y)$, $\mathscr{F}\to f$ iff, whenever $F(X)\ni\mathscr{G}\to_X x$, we have $\ev_*(\mathscr{F}\times \mathscr{G})\to_Y f(x)$. Here $\mathscr{F}\times \mathscr{G}$ is the filter generated by the collection $\{F\times G\mid F\in \mathscr{F}, G\in \mathscr{G}\}$ and $\ev\colon \pstop(X,Y)\times X\to Y$ is the evaluation map.

\begin{defn}[$\epitop$]
A pseudospace $X$ is called epitopological, or an epispace for short, if it is the initial pseudospace\footnote{The original definition by Antoine \cite{Ant1966} \cite{Mac1973} uses the larger category of quasitopological spaces as ``environment'' but it is not hard to prove the equivalence with ours by following a trail of theorems in \cite{Mac1973}.} with respect to a collection of maps $\left\{f_j\colon X\to {Z_j}^{Y_j}\right\}_{j\in J}$ where $Y_j,Z_j\in\top$. The collection of all epispaces determines a full isomorphism-closed subcategory of $\pstop$ denoted by $\epitop$.
\end{defn}

Every topological space $X\in\top$ is epitopological by considering the bijection $X\to X^{\{*\}}$, $x\mapsto (*\mapsto x)$ where $\{*\}$ is a one-point topological space, therefore $\top$ is a full isomorphism-closed subcategory of $\epitop$. Given $Y\in\epitop$, the exponential pseudotopology on $\pstop(X,Y)$ turns out to be epitopological \cite{Mac1973} and thus every epispace is exponentiable. In other words, $\epitop$ is cartesian closed. In a precise sense, $\epitop$ is the smallest cartesian closed topological construct generated by $\top$ (called the cartesian closed topological hull of $\top$, see e.g.\@ \cite{LowSioVer2009} for more details).

We summarize all these results in the next proposition.

\begin{prop}\label{prop:expo}
Of the three topological constructs $$\top\inclu\epitop\inclu\pstop\;,$$ the last two are cartesian closed. For $X,Y\in\pstop$, the exponential pseudotopology defining $Y^X$ is the one described above. Moreover:
\begin{enumerate}
\item\label{expo:itm:first} whenever $Y\in\epitop$, we have $Y^X\in\epitop$. Therefore, the exponential epitopology is just the exponential pseudotopology applied to epispaces;
\item\label{expo:itm:second} whenever $X,Y\in\top$ with $X$ locally compact Hausdorff, we have $Y^X\in\top$. Therefore, for a locally compact Hausdorff space $X$ the exponential topology is just the exponential pseudotopology applied to spaces.
\end{enumerate}
\end{prop}

The next proposition shows that the inclusions $\top\inclu\epitop\inclu\pstop$ admit left adjoints.

\begin{prop}\label{prop:reflectors}
There are concrete reflectors
$$\top\xleftarrow{R_2}\epitop\xleftarrow{R_1}\pstop\;.$$
This means, say for $R_1$, that for each $X\in\pstop$ there is $R_1 X\in\epitop$ such that: $|X|=|R_1 X|$, the identity map $\id\colon X\to R_1 X$ is continuous, and for each continuous map $f\colon X\to Y$ with $Y\in\epitop$ the map $f\colon R_1 X\to Y$ is continuous as well. Analogously for $R_2$ and for the composition $R \coloneqq R_2 R_1$. Moreover, $R_2$ is the restriction of $R$ to $\epitop$, and $R$ admits the following simple description: for $X\in\pstop$, a subset $S\subset |X|$ is open in $R X$ iff whenever $\mathscr{U}\to x\in S$ we have $S\in\mathscr{U}$. 
\end{prop}
\begin{proof}[Sketch of proof]
We proceed backwards: it can be checked that the given description of $R$ defines a concrete reflector, so its restriction $R_2$ to $\epitop$ is also a concrete reflector. The existence of a concrete reflector $R_1$ satisfying $R = R_2 R_1$ relies on the fact that $\epitop$ is initially closed in $\pstop$ \cite{Mac1973} \cite[Prop.\@ 21.31]{AdaHerStr2004} and the fact that all constructs considered here are amnestic \cite{AdaHerStr2004}.
\end{proof}

Notice that each reflector is the identity functor when restricted to its respective image category. By the general theory of topological constructs \cite{AdaHerStr2004} one of the consequences of Prop.\@ \ref{prop:reflectors} is the following.

\begin{cor}\label{cor:preservation}
The inclusions $\top\inclu\epitop\inclu\pstop$ preserve initial sources and the reflectors $\top\xleftarrow{R_2}\epitop\xleftarrow{R_1}\pstop$ preserve final sinks.
\end{cor}

We mention in passing that by \cite{SchWec1992} quotient maps in $\epitop$ (resp., $\pstop$) between topological spaces correspond exactly to product-stable (resp., pullback-stable) quotient maps in $\top$.

\section{A pasting lemma in \texorpdfstring{$\pstop$}{PsTop} and \texorpdfstring{$\epitop$}{EpiTop}}\label{sec:pasting}
In this section we state and prove a generalization to $\pstop$ (and hence also to $\epitop$) of the pasting lemma in $\top$ which we now recall (see e.g.~\cite[Chap~III, Thm~9.4]{Dug1966}).

\begin{lemma}[Pasting Lemma in $\top$]\label{lemma:pasting_top}
Let $X$ be a topological space and $\{X_j\}_{j\in J}$ a cover of $X$ such that either
\begin{enumerate}
\item all $X_j$ are open, or
\item all $X_j$ are closed, and form a locally finite family\footnote{We recall that a family of subsets of a topological space is locally finite if each $x\in X$ has a neighborhood intersecting only finitely many of them.}.
\end{enumerate}
If $Y$ is a topological space and $f\colon X\to Y$ is a function such that each restriction $f_j\colon X_j\to Y$ is continuous, where each $X_j$ carries the subspace topology, then $f$ is continuous.
\end{lemma}

\begin{lemma}[Pasting Lemma in $\pstop$]\label{lemma:pasting_pstop}
Let $X$ be a pseudospace and $\{X_j\}_{j\in J}$ a cover of $X$ such that either
\begin{enumerate}
\item \label{pasting:itm:first} all $X_j$ are open in the reflected topological space $R X$, or
\item \label{pasting:itm:second} all $X_j$ are closed in $R X$, and form a locally finite family in $R X$.
\end{enumerate}
If $Y$ is a pseudospace and $f\colon X\to Y$ is a function such that each restriction $f_j\colon X_j\to Y$ is continuous, where each $X_j$ carries the subspace pseudotopology, then $f$ is continuous.
\end{lemma}

\begin{rmk}\label{rmk:pasting_epi}
Lemma \ref{lemma:pasting_pstop} applies verbatim to $\epitop$ by Cor.\@ \ref{cor:preservation}. When the pseudotopologies of $X$ and $Y$ are topological, we recover Lemma \ref{lemma:pasting_top}.
\end{rmk}

Before proving Lemma \ref{lemma:pasting_pstop} we need to recall the notion of pullback for filters. In contrast to pushforwards, the pullback of a filter does not always exist. However, we have the following.

\begin{defn}
Let $f\colon X\to Y$ be a map between sets and let $\mathscr{F}$ be a filter on $Y$. If $f^{-1}(F)\neq\emptyset$ for each $F\in\mathscr{F}$ then we define $f^*\mathscr{F}$ as the filter generated by the filter base $\{f^{-1}(F)\mid F\in\mathscr{F}\}$.
\end{defn}

\begin{lemma}\label{lemma:pullback}
When defined, the filter $f^*\mathscr{F}$ satisfies the formula $\mathscr{F}\subset f_*f^*\mathscr{F}$. In particular, if $\mathscr{F}$ is an ultrafilter and $f^*\mathscr{F}$ is defined then $f_*f^*\mathscr{F}=\mathscr{F}$, which implies $f(X)\in\mathscr{F}$.
\end{lemma}
\begin{proof}
We compute: $f_*f^*\mathscr{F}=\{S\subset Y\mid f^{-1}(S)\supset f^{-1}(F)\text{ for some }F\in\mathscr{F}\}$. Obviously $\mathscr{F}\subset f_*f^*\mathscr{F}$. Notice that $f(X)\in f_*f^*\mathscr{F}$ since $f^{-1}f(X)=X$. When $\mathscr{F}$ is an ultrafilter, by the maximality of $\mathscr{F}$ we have equality.
\end{proof}

\begin{proof}[Proof of Lemma \ref{lemma:pasting_pstop}]
We first assume that all $X_j$ are open in $R X$. Take $\mathscr{U}\to x$ with $x\in X_j$ for some $j$. By openness we have $X_j\in\mathscr{U}$, so every $U\in\mathscr{U}$ has non-empty intersection with $X_j$ and by Lemma \ref{lemma:pullback} applied to the inclusion map $j\colon X_j\hookrightarrow X$ we have $j_* j^*\mathscr{U}=\mathscr{U}$. Therefore $j^*\mathscr{U}\to x$ in $X_j$ by the definition of subspace pseudotopology, and by the continuity of $f_j$ we have ${f_j}_*(j^*\mathscr{U})=f_*(j_* j^*\mathscr{U})=f_*(\mathscr{U})\to f(x)$ which proves $f$ is continuous. This proves the lemma for case \eqref{pasting:itm:first}. Now we assume case \eqref{pasting:itm:second}, that is, all $X_j$ are closed in $R X$ and form a locally finite family in $R X$. Take $\mathscr{U}\to x$. By assumption there is an open set $V$ containing $x$ and covered by finitely many closed sets $X_j$, call them $X_{j_1}, X_{j_2}, \dots, X_{j_n}$. We now show by contradiction that at least one among $X_{j_1}, \dots, X_{j_n}$ intersects each $U\in\mathscr{U}$: indeed, suppose that for each $i=1,\dots, n$ there is $U_i\in\mathscr{U}$ such that $X_{j_i}\cap U_i=\emptyset$. Then $(\cap_{i=1}^n U_i)\cap (\cup_{i=1}^n X_{j_i})=\emptyset$ which implies $(\cap_{i=1}^n U_i)\cap V=\emptyset$ since $V$ is covered by $\cup_{i=1}^n X_{j_i}$. But $\cap_{i=1}^n U_i$ belongs to $\mathscr{U}$ as a finite intersection of elements of $\mathscr{U}$, and $V$ belongs to $\mathscr{U}$ because $\mathscr{U}\to x\in V$ with $V$ open, therefore $(\cap_{i=1}^n U_i)\cap V\neq\emptyset$ and we have the desired contradiction. In order to fix notation, say $X_k$ has non-empty intersection with all $U\in\mathscr{U}$. By Lemma \ref{lemma:pullback} we have $k_* k^*\mathscr{U}=\mathscr{U}$ and $X_k\in\mathscr{U}$. By closedness $x\in X_k$, so $k^*\mathscr{U}\to x$ in $X_k$. Finally, by the continuity of $f_k$ we have ${f_k}_*(k^*\mathscr{U})=f_*(k_* k^*\mathscr{U})=f_*(\mathscr{U})\to f(x)$ which proves $f$ is continuous.
\end{proof}

\begin{rmk}
One is tempted to extend Lemma \ref{lemma:pasting_pstop} to larger cartesian closed topological constructs such as\footnote{These are constructs consisting of sets equipped with a choice of filters for each point, much like the formulation of $\pstop$ in terms of filters except that weaker axioms are assumed for the choice of filters, see e.g.\@ \cite{ColLow2001}.} $\Lim$ or $\conv$, although our proof seems to suggest that the fact that a pseudotopology can be specified in terms of \emph{ultra}filters is crucial. In \cite[App.\@ A.2]{Pre2002} a sort of pasting lemma is proved for functions from a pretopological space to a limit space, allowing for a definition of the fundamental group in $\Lim$ (and, by recurrence, of all higher homotopy groups). However, that lemma does not ensure the continuity of loop concatenation in a loop limit space (for that we would need a pasting lemma for functions between limit spaces) and thus it does not lead to a notion of a fundamental group enriched over $\Lim$. Whether such a notion exists remains an open problem. On the other hand, the same appendix offers a counterexample to the possibility of extending any form of pasting lemma to $\conv$ (called $\kconv$ in \cite{Pre2002}) by exhibiting a convergence space $X$ for which the path concatenation of two ``path composable'' continuous maps $[0,1]\to X$ is not continuous. In other words, in $\conv$ the unit interval is not suitable to define a reasonable notion of fundamental group.
\end{rmk}

\begin{problem}
Extending Lemma \ref{lemma:pasting_pstop} to $\Lim$, or finding a counterexample to such an extension.
\end{problem}

\section{Path-component functors in \texorpdfstring{$\pstop$}{PsTop} and \texorpdfstring{$\epitop$}{EpiTop}}
We briefly recall that in $\top$ one can define the path-component endofunctor $\pctop\colon\top\to\top$ which assigns:
\begin{itemize}
\item to each space $X$ the quotient space $\pctop X$ of its path components, where the quotient topology is taken with respect to the natural projection $\quot{X}{}$ of $X$ onto the set of its path components;
\item to each continuous map $f\colon X\to Y$ between spaces the continuous map $\pctop f\colon\pctop X\to\pctop Y$ defined by $(\pctop f)(\quot{X}{}(x))\coloneqq \quot{Y}{}(f(x))$.
\end{itemize}
For each continuous map $f\colon X\to Y$ between spaces we thus have the commutative diagram
\begin{center}
\begin{tikzcd}
X\arrow[r,"f"] \arrow[d,"\quot{X}{}"]& Y\arrow[d,"\quot{Y}{}"]\\
\pctop X \arrow[r,"\pctop f"] & \pctop Y
\end{tikzcd}
\end{center}
where $\quot{X}{}$ and $\quot{Y}{}$ are quotient maps in $\top$. Given a product $\Prod_j X_j$ of spaces, there is a natural continuous bijection $\pctop \Prod_j X_j\to \Prod_j {\pctop X_j}$ which is a homeomorphism iff the product of quotient maps $\Prod_j \quot{X_j}{}\colon \Prod_j X_j\to \Prod_j \pctop X_j$ is quotient. Since in $\top$ products of quotient maps need not be quotient, it follows that $\pctop$ does not preserve products (not even finite ones).
See \cite{Bra2012} for more details.

Substituting $\top$ with $\epitop$ and $\pstop$ in the construction above, we obtain analogous endofunctors
\begin{equation*}
\begin{aligned}
\pcepi&\colon\epitop\to\epitop\\
\pcps&\colon\pstop\to\pstop
\end{aligned}
\end{equation*}
with corresponding commutative diagrams
\begin{center}
\begin{tikzcd}
X\arrow[r,"f"] \arrow[d,"\quot{X}{epi}"]& Y\arrow[d,"\quot{Y}{epi}"]&&Z\arrow[r,"g"] \arrow[d,"\quot{Z}{ps}"]& W\arrow[d,"\quot{W}{ps}"]\\
\pcepi X \arrow[r,"\pcepi f"] & \pcepi Y&&\pcps Z \arrow[r,"\pcps g"] & \pcps W
\end{tikzcd}
\end{center}
where on the left-hand side $f\colon X\to Y$ is a continuous map between epispaces and $\quot{X}{epi}$ and $\quot{Y}{epi}$ are quotient maps in $\epitop$, while on the right-hand side $g\colon Z\to W$ is a continuous map between pseudospaces and $\quot{Z}{ps}$ and $\quot{W}{ps}$ are quotient maps in $\pstop$.

The ``convenient'' properties of $\epitop$ and $\pstop$ induce better-behaved path-component functors at least with respect to products, as the next proposition shows.

\begin{prop}\label{prop:pc_products}
$\pcps$ preserves products and $\pcepi$ preserves finite products.
\end{prop}
\begin{proof}
In $\pstop$ products of quotient maps are quotient, see \cite[Theorem~31]{BenHerLow1991}. In $\epitop$ finite products of quotient maps are quotient: to prove it use the characterization of cartesian closedness for topological constructs given in \cite{Her1974} or \cite[Thm 2.14]{HerColSch1991}, together with the fact that $\epitop$ is cartesian closed.
\end{proof}

Let $X$ be a pseudospace and $\gamma\colon[0,1]\to X$ be a path in $X$. Then $R_1\gamma$ is a path in the reflected epispace $R_1 X$ and $R\gamma$ is a path in the reflected space $R X$. By Cor.\@ \ref{cor:preservation} we obtain natural continuous surjections $\quott{X}{ps}\colon R_1 \pcps X \twoheadrightarrow\pcepi R_1 X$ and $\quott{X'}{epi} \colon R_2 \pcepi X'\twoheadrightarrow \pctop R_2 X'$ for each $X\in\pstop$ and each $X'\in\epitop$. The commutative diagram in Fig.\@ \ref{fig:wardrobe} illustrates the mutual relations among $\pcps$, $\pcepi$ and $\pctop$.

\begin{center}
\begin{figure}
\begin{tikzcd}[row sep={2em,between origins}, column sep={2em,between origins}, text height=1.5ex, text depth=0.25ex, nodes in empty cells]
X \arrow[rrrrrr,"f"] \arrow[dddrr,"\id",sloped] \arrow[dddd,swap,"\quot{X}{ps}",->>] &&&&&& Y \arrow[dddd,swap,"\quot{Y}{ps}",densely dotted,->>] \arrow[dddrr,"\id",sloped] \\
\\
\\
&& R_1 X \arrow[rrrrrr,crossing over,"f" near start] &&&&&& R_1 Y \arrow[dddrr,"\id",sloped] \arrow[dddd,"\quot{Y}{ps}",densely dotted,->>] \\
\pcps X \arrow[rrrrrr, "\pcps f" near end,densely dotted,swap] \arrow[dddrr,"\id",sloped] \arrow[dddddddrr,bend right=12,"\quott{X}{ps}",swap,->>]&&&&&& \pcps Y \arrow[dddddddrr,bend right=12,"\quott{Y}{ps}",swap,densely dotted,->>]\arrow[dddrr, "\id" near start,densely dotted,sloped]\\
\\
&&&& R X \arrow[uuull,<-,crossing over,"\id",swap,sloped] \arrow[rrrrrr,crossing over,"f" very near start] &&&&&& R Y \arrow[dddd,"\quot{Y}{ps}",->>]\\
&& R_1\pcps X \arrow[uuuu,<<-,crossing over,"\quot{X}{ps}",swap] \arrow[rrrrrr,crossing over,"\pcps f",densely dotted,swap] \arrow[dddrr,"\id" near start,sloped]  \arrow[dddd,"\quott{X}{ps}",->>] &&&&&& R_1 \pcps Y \arrow[dddrr,"\id",densely dotted,sloped] \arrow[dddd,"\quott{Y}{ps}",densely dotted,->>] \\
\\
\\
&&&& R \pcps X \arrow[uuuu,<<-,crossing over,"\quot{X}{ps}",swap] \arrow[rrrrrr,crossing over,"\pcps f" near start] &&&&&& R \pcps Y \arrow[dddd,"\quott{Y}{ps}",->>] \\
&& \pcepi R_1 X \arrow[rrrrrr,"\pcepi f",densely dotted,swap] \arrow[dddrr,"\id" near start,sloped] \arrow[dddddddrr,bend right=12,"\quott{X}{epi}",swap,->>]&&&&&& \pcepi R_1 Y \arrow[dddrr,"\id",densely dotted,sloped] \arrow[dddddddrr,bend right=12,"\quott{Y}{epi}",swap,densely dotted,->>]\\
\\
\\
&&&& R_2 \pcepi R_1 X \arrow[rrrrrr,crossing over,"\pcepi f"] \arrow[uuuu,<<-,crossing over,"\quott{X}{ps}",swap] \arrow[dddd,"\quott{X}{epi}",->>] &&&&&& R_2 \pcepi R_1 Y \arrow[dddd,"\quott{Y}{epi}",->>]\\
\\
\\
\\
&&&& \pctop R X \arrow[rrrrrr,"\pctop f"]&&&&&& \pctop R Y
\end{tikzcd}
\caption{Mutual relations among $\pcps$, $\pcepi$ and $\pctop$.}\label{fig:wardrobe}
\end{figure}
\end{center}

\begin{problem}
Understanding for which $X\in\pstop$ the map $\quott{X}{ps}$ is quotient in $\epitop$, and analogously for which $X'\in\epitop$ the map $\quott{X'}{epi}$ is quotient in $\top$.
\end{problem}

The next proposition shows that $\pcps$ and $\pcepi$ are lifts of $\pctop$ to $\pstop$ and to $\epitop$, respectively.
\begin{prop}
The following diagram of functors is commutative:
\begin{center}
\begin{tikzcd}
\pstop\arrow[r,"\pcps"] & \pstop\arrow[d,"R_1"]\\
\epitop \arrow[u,hook] \arrow[r,"\pcepi"] & \epitop\arrow[d,"R_2"]\\
\top \arrow[u,hook] \arrow[r,"\pctop"] & \top
\end{tikzcd}
\end{center}
\label{prop:lift}
\end{prop}
\begin{proof}
When $X\in\top$ the map $\quott{X}{epi}$ becomes the identity and we obtain the identification $R_2 \pcepi X = \pctop X$. Analogously, when $X'\in\epitop$ we obtain $R_1 \pcps X' = \pcepi X'$.
\end{proof}

It is natural to ask on which $X\in\top$ these new functors coincide with the usual $\pctop$ (or, what is the same by Prop.\@ \ref{prop:lift}, on which $X\in\top$ these new functors take value in $\top$). Notice that, by Prop.\@ \ref{prop:lift}, $\pcps X\in\top$ implies $\pcepi X\in\top$. For $\pcps$ there is an interesting characterization.
\begin{prop}
For each $X\in\top$ the following are equivalent:
\begin{enumerate}
\item \label{pc:itm:first} $\pcps X\in\top$,
\item \label{pc:itm:second} $\pcps X = \pctop X$,
\item \label{pc:itm:third} the projection $\quot{X}{}\colon X\to\pctop X$ is biquotient.
\end{enumerate}
\end{prop}
\begin{proof}
$\eqref{pc:itm:first} \iff \eqref{pc:itm:second}$ is clear. For $\eqref{pc:itm:second} \iff \eqref{pc:itm:third}$, observe that the equation $\pcps X=\pctop X$ means that on the set of path components of $X$ the quotient pseudotopology and the quotient topology coincide. By a result of Kent \cite[Theorem 5]{Ken1969} this is equivalent to the projection $\quot{X}{}\colon X\to \pctop X$ being biquotient.
\end{proof}
We recall that a continuous surjection $f\colon X\to Y$ between spaces is said to be biquotient if, whenever $y\in Y$ and $\mathcal{O}$ is a covering of $f^{-1}(y)$ by open sets of $X$, then finitely many $f(O)$, with $O\in \mathcal{O}$, cover some neighborhood of $y$ in $Y$. Spaces for which $\quot{X}{}$ is biquotient include semilocally 0-connected spaces (each path component is open) and totally path disconnected spaces (each path component is a singleton). However, these examples are not exhaustive: the topologist's sine curve has biquotient projection but is neither semilocally 0-connected nor totally path disconnected.

\begin{problem}
Finding a topological characterization of those spaces $X$ for which the projection $\quot{X}{}\colon X\to \pctop X$ is biquotient.
\end{problem}

\begin{problem}
Characterizing those spaces $X$ for which $\pcepi X = \pctop X$.
\end{problem}

\begin{prop}\label{prop:cont_mult}
Let $X\in\pstop$ (resp., $X\in\epitop$) and $m\colon X\times X\to X$ be a continuous map. Then the induced map $\mu\colon \pcps X\times \pcps X\to \pcps X$ (resp., $\mu\colon\pcepi X\times\pcepi X\to\pcepi X$) is continuous as well.
\end{prop}
\begin{proof}
It follows from the fact that in both $\pstop$ and $\epitop$ the product of two quotient maps is quotient.
\end{proof}

\begin{rmk}
In $\top$ the above result does not hold and this is the reason why $\biss$ takes value in the category of quasitopological groups, as opposed to topological groups. See the works of Brazas.
\end{rmk}

It is clear that the above treatment holds for the pointed versions of $\pcps$, $\pcepi$ and $\pctop$ as well.

\section{Loop functors in \texorpdfstring{$\pstoppt$}{PsTop*} and \texorpdfstring{$\epitoppt$}{EpiTop*}}
We briefly recall that in the category $\toppt$ of pointed topological spaces and based continuous maps one can define the loop space endofunctor $\looptop\colon\toppt\to\toppt$ which assigns:
\begin{itemize}
\item to each pointed space\footnote{When there is no danger of confusion, for ease of notation we avoid to write the basepoint and thus indicate $(X,x_0)$ simply by $X$.} $(X,x_0)\in\toppt$ the pointed space of based loops in $X$ equipped with the subspace topology inherited from the compact-open topology on the set $\top(S^1,X)$ of all free loops (the basepoint of $\looptop (X,x_0)$ is the constant loop based at $x_0$);
\item to each based continuous map $f\colon (X,x_0)\to (Y,y_0)$ between pointed spaces the based continuous map $\looptop f\colon\looptop (X,x_0)\to\looptop (Y,y_0)$ defined by $(\looptop f)(l)\coloneqq f\circ l$.\end{itemize}

Concatenation and inversion of loops give $\looptop (X,x_0)$ a natural H-group structure, that is, each loop space is a group up to homotopy. More precisely, we adopt the following definition of an H-group.

\begin{defn}[H-group]\label{defn:Hgroup}
An H-group structure on a pointed topological space $(X,x_0)$ consists of pointed continuous maps $\wedge\colon X\times X\to X$ and $\sigma\colon X\to X$ such that:
\begin{enumerate}
\item the maps $x\mapsto x\wedge x_0$ and $x\mapsto x_0\wedge x$ are pointed homotopic to the identity map $\id\colon X\to X$; 
\item the maps $x\mapsto x\wedge \sigma(x)$ and $x\mapsto \sigma(x)\wedge x$ are pointed homotopic to the constant map $x\mapsto x_0$;
\item the map $(x,x',x'')\mapsto (x\wedge x')\wedge x''$ is pointed homotopic to the map $(x,x',x'')\mapsto x\wedge (x'\wedge x'')$.
\end{enumerate}
\end{defn}

The construction of the loop space functor and the definition of an H-group make sense in any cartesian closed topological construct containing $\top$. In particular, we obtain functors
\begin{equation*}
\begin{aligned}
\loopepi&\colon\epitoppt\to\epitoppt\\
\loopps&\colon\pstoppt\to\pstoppt
\end{aligned}
\end{equation*}
and we can ask whether each loop epi/pseudospace is an H-group. The answer is affirmative, as the next proposition shows. We first record two easy results we shall need.

\begin{lemma}\label{lemma:eval}
For each $X\in\pstoppt$ the evaluation map ${\loopps X}\times [0,1]\to X$ is continuous. Analogously for $\epitoppt$ and $\toppt$.
\end{lemma}
\begin{proof}
Write it as ${\loopps X}\times[0,1]\hookrightarrow X^{[0,1]}\times[0,1]\xrightarrow{\ev}X$.
\end{proof}

\begin{lemma}\label{lemma:dual}
For each $X,Y\in\pstoppt$, all set-theoretic maps $h\colon X\times[0,1]\to Y$ such that, for each $x\in X$, $h(x,0)=h(x,1)=y_0$ correspond bijectively to all set-theoretic maps $\hat{h}\colon X\to \loopps(Y,y_0)$. Moreover, $h$ is continuous iff $\hat{h}$ is. Analogously for $\epitoppt$ and $\toppt$.
\end{lemma}
\begin{proof}
Just observe that for a continuous map $X\to Y^{[0,1]}$ with image in $\loopps(Y)$ the restriction of the codomain to $\loopps(Y)$ gives a continuous map $X\to\loopps(Y)$ (we use the fact that the subspace pseudotopology is initial with respect to the inclusion map).
\end{proof}

\begin{prop}\label{prop:Hgroup}
For each $(X,x_0)\in\pstoppt$ and each $(X',x'_0)\in\epitoppt$, concatenation and inversion of loops give $\loopps(X,x_0)$ and $\loopepi(X',x'_0)$ the structure of H-groups.
\end{prop}
\begin{proof}
The proof is a simple adaptation of the usual proof \cite[Chap.\@ IV]{Ser1951} in $\top$, with some care needed when proving continuity of loop concatenation, but for clarity we write it in its entirety. We use symbols $l\cdot l'$ and $l^{-1}$ for loop concatenation and inversion, respectively, and for ease of notation we omit writing the basepoint. We consider $\pstop$ only, the case $\epitop$ being perfectly analogous. To prove the continuity of loop concatenation let us consider $\Phi\colon({\loopps X})^2\times[0,1]\to X$, $\Phi(l,l',t)\coloneqq(l\cdot l')(t)$. Put momentarily $T\coloneqq ({\loopps X})^2$ to save typographic space. The map $\Phi$ restricted to sets $T\times[0,\frac{1}{2}]$ and $T\times[\frac{1}{2},1]$ becomes essentially (up to a continuous rescaling and an uninfluential ${\loopps X}$ factor) the evaluation map ${\loopps X}\times[0,1]\to X$ which is continuous by Lemma~\ref{lemma:eval}. Moreover, $T\times[0,\frac{1}{2}]$ and $T\times[\frac{1}{2},1]$ are closed in $R(T)\times[0,1]$, and the identity map $R(T\times[0,1])\to R(T)\times[0,1]$ is continuous. Therefore $T\times[0,\frac{1}{2}]$ and $T\times[\frac{1}{2},1]$ are closed in $R(T\times[0,1])$ as well and by the pasting lemma in $\pstop$, Lemma \ref{lemma:pasting_pstop}, we deduce that $\Phi$ is continuous. Now observe that $\Phi(l,l',0)=l(0)=x_0=l'(1)=\Phi(l,l',1)$ so we can apply Lemma~\ref{lemma:dual} and conclude that loop concatenation $\hat{\Phi}\colon({\loopps X})^2\to {\loopps X}$ is continuous. The proof for loop inversion is similar (but it does not use the pasting lemma). We now prove the remaining properties in the definition of an H-group. Let us call $e$ the constant loop at $x_0$.
\begin{enumerate}
\item to prove that the map $l\mapsto l\cdot e$ is pointed homotopic to the identity map, consider the continuous map $\phi\colon[0,1]^2\to[0,1]$, $\phi(s,t)\coloneqq(1-s)\min(2t,1)+st$, then take $\id\times\phi\colon{\loopps X}\times[0,1]^2\to {\loopps X}\times[0,1]$ and compose it with the evaluation obtaining the continuous map $\sigma\colon{\loopps X}\times[0,1]^2\to X$, $\sigma(l,s,t)\coloneqq l(\phi(s,t))$. Observe that $\sigma(l,s,0)=l(0)=x_0=l(1)=\sigma(l,s,1)$ so that by Lemma~\ref{lemma:dual} we get a continuous map $\hat{\sigma}\colon{\loopps X}\times[0,1]\to{\loopps X}$. Finally observe that $\hat{\sigma}(l,0)=l\cdot e$, $\hat{\sigma}(l,1)=l$ and $\hat{\sigma}(e,s)=e$. Substituting $\min$ with $\max$ proves the analogous statement for the map $l\mapsto e\cdot l$;
\item to prove that the map $l\mapsto l\cdot l^{-1}$ is pointed homotopic to the constant map $l\mapsto e$, proceed as above with the continuous map $\psi\colon[0,1]^2\to[0,1]$, $\psi(s,t)\coloneqq 2(1-s)\min(t,1-t)$; substituting $\min$ with $\max$ proves the analogous statement for the map $l\mapsto l^{-1}\cdot l$;
\item to prove that the map $(l,l',l'')\mapsto (l\cdot l')\cdot l''$ is pointed homotopic to the map $(l,l',l'')\mapsto l\cdot (l'\cdot l'')$, proceed as above with the continuous map $\chi\colon[0,1]^2\to[0,1]$ given by
\begin{equation*}
\chi(s,t)=\begin{cases}
\frac{t}{1+s} & 0\leq t\leq\frac{1+s}{4}\\
t-\frac{s}{4} & \frac{1+s}{4}\leq t\leq\frac{2+s}{4}\\
\frac{1}{2}+\frac{4t-2-s}{4-2s} & \frac{2+s}{4}\leq t\leq 1
\end{cases}
\end{equation*}
(notice that this time we shall need the previously established continuity of $\cdot$ to let the argument go through).
\end{enumerate}
\end{proof}

\begin{rmk}
The above proof shows that, given a cartesian closed topological construct $\C$ containing $\top$ as a subconstruct, for each $X\in\Cpt$ the loop object ${\loopC X}$ satisfies all the properties of an H-group except possibly continuity of loop concatenation and homotopy associativity.
\end{rmk}

In the next proposition we record a useful property of $\loopepi$ and $\loopps$ that can be proven in exactly the same way as for $\looptop$.
\begin{prop}\label{prop:loop_products}
Both $\loopepi$ and $\loopps$ preserve arbitrary products. 
\end{prop}

Let $X$ be a pseudospace and $l\colon[0,1]\to X$ be a loop in $X$. Then $R_1 l$ is a loop in the reflected epispace $R_1 X$ and $R l$ is a loop in the reflected space $R X$. Given that the reflectors do not modify the underlying set-theoretic maps we can identify ${\loopps X}$ with a subset of ${\loopepi X}$ and analogously ${\loopepi X}\hookrightarrow{\looptop X}$. These inclusions are continuous by Prop.\@ \ref{prop:expo}. The commutative diagram in Fig.\@ \ref{fig:wardrobe2} illustrates the mutual relations among $\loopps$, $\loopepi$ and $\looptop$.

\begin{center}
\begin{figure}
\begin{tikzcd}[row sep={2em,between origins}, column sep={2em,between origins}, text height=1.5ex, text depth=0.25ex, nodes in empty cells]
X \arrow[rrrrrr,"f"] \arrow[dddrr,"\id",sloped] &&&&&& Y \arrow[dddrr,"\id",sloped] \\
\\
\\
&& R_1 X \arrow[rrrrrr,crossing over,"f" near start] &&&&&& R_1 Y \arrow[dddrr,"\id",sloped] \\
\loopps X \arrow[rrrrrr, "\loopps f" near end,densely dotted,swap] \arrow[dddrr,"\id",sloped] \arrow[dddddddrr,bend right=12,swap,hook]&&&&&& \loopps Y \arrow[dddddddrr,bend right=12,swap,densely dotted,hook]\arrow[dddrr, "\id" near start,densely dotted,sloped]\\
\\
&&&& R X \arrow[uuull,<-,crossing over,"\id",sloped,near start] \arrow[rrrrrr,crossing over,"f" very near start] &&&&&& R Y \\
&& R_1\loopps X \arrow[rrrrrr,crossing over,"\loopps f",swap] \arrow[dddrr,"\id",sloped]  \arrow[dddd,hook] &&&&&& R_1 \loopps Y \arrow[dddrr,"\id",sloped] \arrow[dddd,hook,densely dotted] \\
\\
\\
&&&& R \loopps X \arrow[rrrrrr,crossing over,"\loopps f" near start] &&&&&& R \loopps Y \arrow[dddd,hook] \\
&& \loopepi R_1 X \arrow[rrrrrr,"\loopepi f",densely dotted,swap] \arrow[dddrr,"\id",sloped] \arrow[dddddddrr,bend right=12,swap,hook]&&&&&& \loopepi R_1 Y \arrow[dddrr,"\id",densely dotted,sloped] \arrow[dddddddrr,bend right=12,swap,densely dotted,hook]\\
\\
\\
&&&& R_2 \loopepi R_1 X \arrow[rrrrrr,crossing over,"\loopepi f"] \arrow[uuuu,hookleftarrow,crossing over,swap] \arrow[dddd,hook] &&&&&& R_2 \loopepi R_1 Y \arrow[dddd,hook]\\
\\
\\
\\
&&&& \looptop R X \arrow[rrrrrr,"\looptop f"]&&&&&& \looptop R Y
\end{tikzcd}
\caption{Mutual relations among $\loopps$, $\loopepi$ and $\looptop$.}\label{fig:wardrobe2}
\end{figure}
\end{center}

\begin{problem}
Understanding for which $(X,x_0)\in\pstoppt$ the inclusion
$$R_1 \loopps(X,x_0)\hookrightarrow \loopepi(R_1 X,x_0)$$
is initial in $\epitop$, that is, for which pointed pseudospaces $(X,x_0)$ the epispace $R_1 \loopps(X,x_0)$ is an epitopological subspace of $\loopepi(R_1 X,x_0)$. Analogous question for the inclusion $R_2 \loopepi(X',x'_0)\hookrightarrow \looptop(R_2 X',x'_0)$ where $(X',x'_0)\in\epitoppt$.
\end{problem}

By the above discussion and by Prop.\@ \ref{prop:expo} it is clear that $\loopps$ (resp., $\loopepi$) is an extension of $\looptop$ to $\pstoppt$ (resp., to $\epitoppt$). We record this fact in the next proposition.
\begin{prop}\label{prop:extension}
If $(X,x_0)\in\epitoppt$ then $\loopps(X,x_0)=\loopepi(X,x_0)$. If $(X',x'_0)\in\toppt$ then $\loopepi(X',x'_0)=\looptop(X',x'_0)$.
\end{prop}

\section{Epitopological and pseudotopological fundamental groups}
We now put together the previous constructions to obtain epi- and pseudo-topologizations of the fundamental group of a pointed epi- or pseudospace (in particular, of a pointed topological space). By Prop.\@ \ref{prop:Hgroup} the set of path components of a loop epi/pseudospace is a group (with multiplication and inversion induced by loop concatenation and loop inversion). On the other hand, the epi/pseudo-topologized path component functors $\pcepi$ and $\pcps$ endow this group with an epi/pseudotopology and by functoriality the induced inversion map is continuous. The continuity of the multiplication follows from Prop.\@ \ref{prop:cont_mult}.

\begin{defn}
A pseudotopological group is a group equipped with a pseudotopology making the inverse map and the multiplication map continuous. The category of all pseudotopological groups and continuous group homomorphisms between them is denoted by $\pstopgrp$. The category $\epitopgrp$ is defined analogously.
\end{defn}

\begin{rmk}
The obvious embeddings $\topgrp\hookrightarrow\epitopgrp\hookrightarrow\pstopgrp$ are full, and each category in this chain of inclusions is a topological concrete category over $\grp$. Moreover, $\pstopgrp$ (resp., $\epitopgrp$, $\topgrp$) is the category of group-objects of $\pstop$ (resp., $\epitop$, $\top$). On the other hand, the category $\qtopgrp$ of quasitopological groups is contained neither in $\epitopgrp$ nor in $\pstopgrp$. Rather, if we consider all these categories as subcategories of the category of groups with pseudotopology (no requirements on the continuity of operations) and continuous homomorphisms, we have:
$$\qtopgrp\cap\pstopgrp=\qtopgrp\cap\epitopgrp=\topgrp\;.$$
\end{rmk}

The above discussion together with Prop.\@ \ref{prop:Hgroup} implies that, for each $X\in\pstoppt$ and each $X'\in\epitoppt$, $\pcps\loopps X$ is a pseudotopological group and $\pcepi\loopepi X$ is an epitopological group. We record this fact in the next definition.

\begin{defn}
The epitopological and pseudotopological fundamental group functors are the composed functors
\begin{equation*}
\begin{aligned}
\piepi=\pcepi\loopepi&\colon\epitoppt\to\epitopgrp\\
\pips=\pcps\loopps&\colon\pstoppt\to\pstopgrp\;.\\
\end{aligned}
\end{equation*}
\end{defn}

The mutual relations among these functors are illustrated by the diagram in Fig.~\ref{fig:wardrobe3}, constructed by taking into account the results of the previous sections.

\begin{center}
\begin{figure}
\begin{tikzcd}[row sep={2em,between origins}, column sep={2em,between origins}, text height=1.5ex, text depth=0.25ex, nodes in empty cells]
X \arrow[rrrrrr,"f"] \arrow[dddrr,"\id",sloped] &&&&&& Y \arrow[dddrr,"\id",sloped] \\
\\
\\
&& R_1 X \arrow[rrrrrr,crossing over,"f" near start] &&&&&& R_1 Y \arrow[dddrr,"\id",sloped] \\
\pips X \arrow[rrrrrr, "\pips f" near end,densely dotted,swap] \arrow[dddrr,"\id",sloped] \arrow[dddddddrr,bend right=12,swap]&&&&&& \pips Y \arrow[dddddddrr,bend right=12,swap,densely dotted]\arrow[dddrr, "\id" near start,densely dotted,sloped]\\
\\
&&&& R X \arrow[uuull,<-,crossing over,"\id",near start,sloped] \arrow[rrrrrr,crossing over,"f" very near start] &&&&&& R Y \\
&& R_1\pips X \arrow[rrrrrr,crossing over,"\pips f",swap] \arrow[dddrr,"\id",sloped]  \arrow[dddd] &&&&&& R_1 \pips Y \arrow[dddrr,"\id",sloped] \arrow[dddd,densely dotted] \\
\\
\\
&&&& R \pips X \arrow[rrrrrr,crossing over,"\pips f" near start] &&&&&& R \pips Y \arrow[dddd] \\
&& \piepi R_1 X \arrow[rrrrrr,"\piepi f",densely dotted,swap] \arrow[dddrr,"\id",sloped] \arrow[dddddddrr,bend right=12,swap]&&&&&& \piepi R_1 Y \arrow[dddrr,"\id",densely dotted,sloped] \arrow[dddddddrr,bend right=12,swap,densely dotted]\\
\\
\\
&&&& R_2 \piepi R_1 X \arrow[rrrrrr,crossing over,"\piepi f"] \arrow[uuuu,leftarrow,crossing over,swap] \arrow[dddd] &&&&&& R_2 \piepi R_1 Y \arrow[dddd]\\
\\
\\
\\
&&&& \biss R X \arrow[rrrrrr,"\biss f"]&&&&&& \biss R Y
\end{tikzcd}
\caption{Mutual relations among $\pips$, $\piepi$ and $\biss$.}\label{fig:wardrobe3}
\end{figure}
\end{center}

By combining Prop.\@ \ref{prop:pc_products} and Prop.\@ \ref{prop:loop_products} we get at once the following result.
\begin{prop}
$\piepi$ preserves finite products and $\pips$ preserves arbitrary products.
\end{prop}

It is also not hard to check that both $\piepi$ and $\pips$ are homotopy invariant, that is, given two pointed-homotopic continuous maps $f,g\colon X\to Y$ between pointed epispaces (resp., pseudospaces) the induced maps are equal, $\piepi f=\piepi g$ (resp., $\pips f=\pips g$). Analogously, one can check that a change of basepoint induces an isomorphism of epitopological groups (resp., pseudotopological groups).

Much like the path component functors $\pcps$ and $\pcepi$ are lifts of $\pctop$, the functors $\pips$ and $\piepi$ are lifts of $\biss$ to $\pstop$ and to $\epitop$, respectively.
\begin{prop}
The following diagrams of functors are commutative:
\begin{center}
\begin{tikzcd}
\epitoppt\arrow[r,"\piepi"] & \epitopgrp\arrow[d,"R_2"]&&\pstoppt\arrow[r,"\pips"] & \pstopgrp\arrow[d,"R"]\\
\toppt \arrow[u,hook] \arrow[r,"\biss"] & \qtopgrp&&\toppt \arrow[u,hook] \arrow[r,"\biss"] & \qtopgrp\;.
\end{tikzcd}
\end{center}
\end{prop}
\label{prop:lift2}
\begin{proof}
It follows by previously established corresponding results for path component and loop functors.
\end{proof}

By the above proposition, the restrictions of $\piepi$ and $\pips$ to $\toppt$ contain no less information than $\biss$. The next problem asks whether they contain more information at all.

\begin{problem}
Determining whether there are non-homeomorphic spaces $X\not\simeq Y$ for which ${\biss X}\simeq {\biss Y}$, ${\brazas X}\simeq{\brazas Y}$ but ${\piepi X}\not\simeq{\piepi Y}$ or ${\pips X}\not\simeq{\pips Y}$.
\end{problem}

\begin{prop}\label{prop:compare}
Let $X\in\toppt$. Then the statements contained in each sublist are equivalent:
\begin{enumerate}
\item \begin{enumerate}
	\item ${\pips X}\in\epitopgrp$,
	\item ${\pips X}={\piepi X}$,
	\item on $|{\pctop{\looptop X}}|$ the quotient pseudotopology and the quotient epitopology coincide\footnote{Here and in the following the notation $|X|$ indicates the underlying set of $X$.}.
	\end{enumerate}
\item \begin{enumerate}
	\item ${\piepi X}\in\qtopgrp$,
	\item ${\piepi X}\in\topgrp$,
	\item ${\piepi X}={\biss X}$,
	\item ${\piepi X}={\brazas X}$,
	\item on $|{\pctop{\looptop X}}|$ the quotient epitopology and the quotient topology coincide.
	\end{enumerate}
\item \begin{enumerate}
	\item ${\pips X}\in\qtopgrp$,
	\item ${\pips X}\in\topgrp$,
	\item ${\pips X}={\biss X}$,
	\item ${\pips X}={\brazas X}$,
	\item on $|{\pctop{\looptop X}}|$ the quotient pseudotopology and the quotient topology coincide,
	\item ${\looptop X}\to {\pctop{\looptop X}}$ is biquotient.
	\end{enumerate}
\end{enumerate}
\end{prop}
\begin{proof}[Sketch of proof]
Use previous results together with the fact: ${\biss X}={\brazas X}$ iff ${\biss X}\in\topgrp$ (see \cite[Prop.~3.3]{Bra2013}).
\end{proof}

The above proposition does not clarify whether ${\biss X}={\brazas X}$ implies ${\piepi X}\in\qtopgrp$ or ${\pips X}\in\qtopgrp$. If either of these implications is true then ${\biss X}$ is a topological group iff $\quot{\looptop X}{}\times\quot{\looptop X}{}$ is quotient, where $\quot{\looptop X}{}\colon \looptop X\to {\pctop{\looptop X}}$ is the natural quotient map in $\top$ (thereby settling a question in \cite[p.17]{BraFab2013}).

\begin{rmk}
Let $X\in\toppt$ be such that ${\looptop X}$ is totally path disconnected. By Prop.\@ \ref{prop:compare}, ${\biss X}={\brazas X}={\pips X}={\piepi X}\simeq{\looptop X}$.
\end{rmk}

Many questions remain to be addressed. We mention the following.

\begin{problem}
Understanding whether $\piepi$ and $\pips$ are essentially surjective, that is, whether for each $G\in\epitopgrp$ and each $G'\in\pstopgrp$ there are an epispace $X$ and a pseudospace $X'$ such that $\piepi X\simeq G$ and $\pips X'\simeq G'$.
\end{problem}

\begin{problem}
Understanding what is the image of $\toppt$ under $\piepi$ and under $\pips$.
\end{problem}

\section{Conclusions}
We identified two supercategories of $\top$, namely $\epitop$ and $\pstop$, where it is possible to define well-behaved enriched fundamental group functors, $\piepi$ and $\pips$, by replicating the quotient construction specified in \cite{Bis2002}. One of the key points, which is a new result to the best of the author's knowledge, is the existence of a pasting lemma within these constructs. The product-stability of quotient maps in these larger categories guarantees that the target category of $\piepi$ and $\pips$ consists of the category of their group-objects, thus avoiding the pitfall of the original construction $\biss$ \cite{Bis2002}. It then turns out that $\piepi$ and $\pips$ are suitable lifts of $\biss$ under the natural inclusions $\top\inclu\epitop\inclu\pstop$ and the corresponding reflectors $\top\xleftarrow{R_2}\epitop$ and $\top\xleftarrow{R}\pstop$, as shown by the following diagrams:
\begin{center}
\begin{tikzcd}
\epitoppt\arrow[r,"\piepi"] & \epitopgrp\arrow[d,"R_2"]&&\pstoppt\arrow[r,"\pips"] & \pstopgrp\arrow[d,"R"]\\
\toppt \arrow[u,hook] \arrow[r,"\biss"] & \qtopgrp&&\toppt \arrow[u,hook] \arrow[r,"\biss"] & \qtopgrp\;.
\end{tikzcd}
\end{center}
These two functors should be compared with the topologized fundamental group introduced and studied in \cite{Bra2013}. In particular, it should be possible to devise a covering epispace theory and a covering pseudospace theory along the lines of \cite{Bra2012a}, as well as constructing groupoid versions of $\piepi$ and $\pips$.
Obviously it is possible to recursively define all higher homotopy group functors by putting $\pihepi\coloneqq \pcepi(\loopepi)^n$ and $\pihps\coloneqq \pcps(\loopps)^n$. In a broader perspective, it would be nice to investigate how much of classical homotopy theory carries over to this new realm and what, if any, new insight the richer structure of these functors might offer.

\bibliographystyle{alpha}
\bibliography{pi1top2017.bbl}

\begin{thebibliography}{Bra12b}

\bibitem[AHS04]{AdaHerStr2004}
Ji{\v r}{\'\i} Ad{\'a}mek, Horst Herrlich, and George~E. Strecker.
\newblock {\em Abstract and Concrete Categories: The Joy of Cats}.
\newblock \url{http://katmat.math.uni-bremen.de/acc}, 2004.

\bibitem[Ant66]{Ant1966}
Philippe Antoine.
\newblock Etude {\'e}l{\'e}mentaire des cat{\'e}gories d'ensembles
  structur{\'e}s.
\newblock {\em Bull. Soc. Math. Belg}, 18:142--164, 1966.

\bibitem[BF13]{BraFab2013}
Jeremy Brazas and Paul Fabel.
\newblock On fundamental groups with the quotient topology.
\newblock {\em Journal of Homotopy and Related Structures}, pages 1--21, 2013.

\bibitem[BHL91]{BenHerLow1991}
Herschel~L. Bentley, Horst Herrlich, and Robert Lowen.
\newblock Improving constructions in topology.
\newblock In {\em Category Theory at Work}, pages 3--20. Heldermann Verlag,
  Berlin, 1991.

\bibitem[Bis02]{Bis2002}
Daniel~K. Biss.
\newblock The topological fundamental group and generalized covering spaces.
\newblock {\em Topology and its Applications}, 124(3):355--371, 2002.

\bibitem[Bou75]{Bou1975}
G{\'e}rard Bourdaud.
\newblock Espaces d'{A}ntoine et semi-espaces d'{A}ntoine.
\newblock {\em Cahiers de Topologie et G{\'e}om{\'e}trie Diff{\'e}rentielle
  Cat{\'e}goriques}, 16(2):107--133, 1975.

\bibitem[Bra11]{Bra2011}
Jeremy Brazas.
\newblock The topological fundamental group and free topological groups.
\newblock {\em Topology and its Applications}, 158(6):779--802, 2011.

\bibitem[Bra12a]{Bra2012a}
Jeremy Brazas.
\newblock Semicoverings: a generalization of covering space theory.
\newblock {\em Homology, Homotopy and Applications}, 14(1):33--63, 2012.

\bibitem[Bra12b]{Bra2012}
Jeremy Brazas.
\newblock The topology of path component spaces.
\newblock Notes available at \url{http://www2.gsu.edu/~jbrazas/pcs.pdf}, 2012.

\bibitem[Bra13]{Bra2013}
Jeremy Brazas.
\newblock The fundamental group as a topological group.
\newblock {\em Topology and its Applications}, 160(1):170--188, 2013.

\bibitem[Cho48]{Cho1947}
Gustave Choquet.
\newblock Convergences.
\newblock {\em Annales de l'Universit{\'e} de Grenoble}, 23:57--112, 1947-1948.

\bibitem[CL01]{ColLow2001}
Eva Colebunders and Robert Lowen.
\newblock Supercategories of {T}op and the inevitable emergence of topological
  constructs.
\newblock In {\em Handbook of the History of General Topology}, volume~3, pages
  969--1026. Springer, 2001.

\bibitem[Dug66]{Dug1966}
James Dugundji.
\newblock {\em Topology}.
\newblock Allyn \& Bacon, 1966.

\bibitem[EH01]{EscHec2001}
Mart{\'\i}n Escard{\'o} and Reinhold Heckmann.
\newblock Topologies on spaces of continuous functions.
\newblock {\em Topology Proceedings}, 26(2):545--564, 2001.

\bibitem[Fab09]{Fab2009}
Paul Fabel.
\newblock Multiplication is discontinuous in the {H}awaiian earring group (with
  the quotient topology).
\newblock {\em ArXiv preprint, arXiv:0909.3086}, 2009.

\bibitem[HCS91]{HerColSch1991}
Horst Herrlich, Eva Colebunders, and Friedhelm Schwarz.
\newblock Improving {T}op: {P}r{T}op and {P}s{T}op.
\newblock In {\em Category Theory at Work}, pages 21--34. Heldermann Verlag,
  Berlin, 1991.

\bibitem[Her74]{Her1974}
Horst Herrlich.
\newblock Cartesian closed topological categories.
\newblock In {\em Math. Colloq. Univ. Cape Town}, volume~9, pages 1--16, 1974.

\bibitem[Her87]{Her1987}
Horst Herrlich.
\newblock Topological improvements of categories of structured sets.
\newblock {\em Topology and its Applications}, 27(2):145--155, 1987.

\bibitem[Ken69]{Ken1969}
Darrell~C. Kent.
\newblock Convergence quotient maps.
\newblock {\em Fundamenta Mathematicae}, 65(2):197--205, 1969.

\bibitem[LSV09]{LowSioVer2009}
Robert Lowen, Mark Sioen, and Stijn Verwulgen.
\newblock Categorical {T}opology.
\newblock In {\em Beyond Topology}, volume 486, pages 1--35. AMS, 2009.

\bibitem[Mac73]{Mac1973}
Armando Machado.
\newblock Espaces d'{A}ntoine et pseudo-topologies.
\newblock {\em Cahiers de Topologie et G{\'e}om{\'e}trie Diff{\'e}rentielle
  Cat{\'e}goriques}, 14(3):309--327, 1973.

\bibitem[Pre02]{Pre2002}
Gerhard Preuss.
\newblock {\em Foundations of Topology: an approach to Convenient Topology}.
\newblock Kluwer Academic Publishers, Dordrecht, 2002.

\bibitem[Ser51]{Ser1951}
Jean-Pierre Serre.
\newblock Homologie singuliere des espaces fibr{\'e}s.
\newblock {\em Annals of Mathematics}, 54(3):425--505, 1951.

\bibitem[SWS92]{SchWec1992}
Friedhelm Schwarz and Sibylle Weck-Schwarz.
\newblock On hereditary and product-stable quotient maps.
\newblock {\em Comment. Math. Univ. Carolin}, 33(2):345--352, 1992.

\bibitem[Wyl91]{Wyl1991}
Oswald Wyler.
\newblock {\em Lecture notes on Topoi and Quasitopoi}.
\newblock World Scientific, 1991.

\end{thebibliography}

\end{document}